\documentclass{elsarticle}

\usepackage{amssymb,amsmath,amsopn, amsthm}
\usepackage{url}

\renewcommand{\theta}{\vartheta}

\newtheorem{theorem}{Theorem}[section]
\newtheorem{lemma}[theorem]{Lemma}
\newtheorem{proposition}[theorem]{Proposition}

\theoremstyle{definition}
\newtheorem*{definition}{Definition}
\newtheorem*{remark}{Remark}

\newcommand{\F}{\mathbf{F}}
\newcommand{\N}{\mathbf{N}}

\newcommand{\PGL}{\mathrm{PGL}} 
\newcommand{\PSL}{\mathrm{PSL}} 
\newcommand{\PSU}{\mathrm{PSU}}  
\newcommand{\GL}{\mathrm{GL}}   
\newcommand{\SL}{\mathrm{SL}}

\DeclareMathOperator{\diag}{diag}

\DeclareMathOperator{\hcf}{hcf}

\DeclareMathOperator{\efs}{efs}
\DeclareMathOperator{\sss}{ss}

\begin{document}
\title{Normal coverings of linear groups}

\author[Bristol]{John R. Britnell}
\ead{j.r.britnell@bristol.ac.uk}

\author[Renyi]{Attila Mar\'oti\fnref{AMack}}
\ead{maroti.attila@renyi.mta.hu}

\fntext[AMack]{The research of the second author was supported by a
   Marie Curie International Reintegration Grant within the~7th
   European Community Framework Programme, by the J\'anos Bolyai
   Research Scholarship of the Hungarian Academy of Sciences, and by
   OTKA K84233.}

\address[Bristol]{School of Mathematics, University of Bristol, University Walk, Bristol,
BS8 1TW,\\ United Kingdom}

\address[Renyi]{Alfr\'ed R\'enyi Institute of Mathematics, Re\'altanoda
   utca 13-15, H-1053, Budapest, Hungary}


\begin{abstract}
For a non-cyclic finite group $G$, let $\gamma(G)$ denote the
smallest number of conjugacy classes of proper subgroups of $G$
needed to cover $G$. Bubboloni, Praeger and Spiga, motivated by questions in number theory, have recently established that $\gamma(S_n)$ and
$\gamma(A_{n})$ are bounded above and below by linear functions of
$n$. In this paper we show that if $G$ is in the range $\SL_{n}(q)\le G\le \GL_{n}(q)$ for $n>2$, then
$n/\pi^2 < \gamma(G) \le (n+1)/2$.
We give various alternative bounds, and derive explicit formulas for $\gamma(G)$ in some cases.
\end{abstract}

\maketitle

\section{Introduction}
\subsection{Normal coverings}

Let $G$ be a non-cyclic finite group. We write
$\gamma(G)$ for the smallest number of conjugacy classes of proper subgroups of $G$ needed to
cover it. In other words, $\gamma(G)$ is
the least $k$ for which there exist subgroups $H_1,\dots,H_k<G$ such that
\[
G=\bigcup_{i=1}^k\, \bigcup_{g\in G} {H_i}^g.
\]
We say that the set of conjugacy classes $\{{H_i}^G\mid i=1,\dots,k\}$ is a \emph{normal covering} for $G$.

Bubboloni and Praeger \cite{BP} have recently investigated $\gamma(G)$ in the case that $G$ is a finite symmetric or alternating group. They show,
for example, that if $n$ is an odd composite number then
\[
\frac{\phi(n)}{2}+1\le\gamma(S_n)\le\frac{n-1}{2},
\]
where $\phi$ is Euler's totient function. Similar results are established for all values of $n$, and for both $S_n$ and $A_n$. Part of the motivation
for their work comes from an application in number theory.

It is a well-known theorem of Jordan that no finite group is covered by the conjugates of any proper subgroup. A paraphrase of this statement is that
$\gamma(G)>1$ for any finite group $G$.
It is known that there exists a solvable group $G$ with $\gamma(G)=k$ for any $k>1$ \cite{CL}.
It has been shown in \cite{BL}
that if $G\in\{\GL_n(q),\SL_n(q),\PGL_n(q),\PSL_n(q)\}$ then $\gamma(G)=2$ if and only if $n\in \{2,3,4\}$.
(Notice that $\gamma$ is undefined for $n=1$, since the groups are cyclic in this case.)
Other groups of Lie type possessing a normal covering of size $2$ have been studied in \cite{BLW1} and \cite{BLW2}.

In this paper we give bounds on $\gamma(G)$, where $\SL_n(q)\le G\le \GL_n(q)$, for all values of $n$. In some
cases we are able to give an exact value. Our bounds extend without change to $G/Z(G)$.

We introduce some notation. We write $\lfloor x\rfloor$ for the integer part of a real number $x$.
As already noted above, $\phi$ denotes Euler's function. We shall also use Lehmer's
\emph{partial totient function}, which we define here.

\begin{definition}
Let $k$ and $q$ be such that $0\le q< k< n$. We define the partial totient
$\phi(k, t, n)$ to be the number of integers $x$, coprime with $n$, such that
\[
\frac{nt}{k} < x < \frac{n(t + 1)}{k}.
\]
\end{definition}

We give two separate upper bounds on $\gamma(G)$.

\begin{theorem}\label{thm:upperbounds}
Let $n\in\N$, and let $\nu=\nu(n)$ be the number of prime factors of $n$.
Let $p_1,\dots,p_{\nu}$ be the distinct prime factors of $n$, with $p_1<p_2<\cdots<p_{\nu}$. Let $G$ be a group
such that $\SL_n(q)\le G\le \GL_n(q)$. Then
\begin{enumerate}
\item If $\nu \ge 2$ then
\[
\gamma(G) \le \left(1-\frac{1}{p_1}\right)\left(1-\frac{1}{p_2}\right)\frac{n}{2}+2.
\]
\item If $n>6$ then
\[
\gamma(G) \le
\left\lfloor{\frac{n}{3}}\right\rfloor + \phi(6,2,n) +\nu.
\]
\end{enumerate}
\end{theorem}

A great deal of information is given in \cite[\S6]{L} about the function $\phi(6,t,n)$, from which the following
statement can be derived.
\[
\frac{\phi(n)}{6}-\phi(6,2,n) =\left\{
\begin{array}{ll}
0 & \textrm{if $n$ is divisible either by $9$, or by a prime}\\
&\quad\textrm{of the form $3k+1$ for $k\in\N$,} \\
\frac{1}{12}\lambda(n)2^\nu & \textrm{otherwise, if $n$ is divisible by $3$,} \\ \\
\frac{1}{6}\lambda(n)2^\nu & \textrm{otherwise, if $n$ is not divisible by $3$,}
\end{array}\right.
\]
in which $\lambda(n)=(-1)^\ell$, where
$\ell$ is the number of prime divisors of $n$ counted with multiplicity.

\subsection{Independent sets of conjugacy classses}

Let $\kappa(G)$ be the size of the largest set of conjugacy classes of $G$ such that any pair of elements from distinct classes generates $G$.
We call such a set an \emph{independent set of classes}.
Guralnick and Malle \cite{GM} have shown that $\kappa(G)\ge 2$ for any finite simple group $G$. It is clear that whenever
$\gamma(G)$ is defined, we have the inequality
\[
\kappa(G)\le \gamma(G),
\]
since if $\mathcal{C}$ is a
normal covering of $G$, and if $\mathcal{I}$ is a independent set of classes, then each element of
$\mathcal{C}$ covers at most one element of $\mathcal{I}$.

We establish two lower bounds for $\kappa(G)$. By the observation of the previous paragraph, these also operate as
lower bounds for $\gamma(G)$.

\begin{theorem}\label{thm:lowerbounds}
Let $n\in\N$, and let $\nu=\nu(n)$ be the number of prime factors of $n$. Let $p_1,\dots,p_{\nu}$ be the distinct
prime factors of $n$, with $p_1<p_2<\cdots<p_{\nu}$. Let $G$ be a group such that $\SL_n(q)\le G\le \GL_n(q)$.
\begin{enumerate}
\item If $\nu\ge 2$ then
\[
\frac{\phi(n)}{2} + \nu(n) \le \kappa(G).
\]
\item If $\nu\ge 3$, and if $n$ is not equal to $6p$ or $10p$ for any prime $p$, then
\[
\left\lfloor\frac{n+6}{12}\right\rfloor + \phi(12,1,3n) + \nu \le \kappa(G).
\]
Furthermore, if $\hcf(n,6)=1$ then
\[
\left\lfloor\frac{n+6}{12}\right\rfloor + \phi(12,1,3n) + \phi(12,0,n) + \nu \le \kappa(G).
\]
\end{enumerate}
\end{theorem}

The values $t=0,1$ are not amongst those for which the function $\phi(12,t,n)$ is evaluated explicitly in \cite{L}.
However, Theorem 10 of \cite{L} gives the following general estimate,
\[
\left|\phi(n)-k\phi(k,t,n)\right|\le (k-1)2^\nu,
\]
where $\nu$ is the number of prime divisors of $n$. This yields the lower bound
\[
\phi(12,t,n) \ge \frac{\phi(n)}{12} - \frac{11}{12}2^\nu.
\]

There are certain cases in which an upper bound for $\gamma(G)$ coincides with a lower bound for $\kappa(G)$. In
these cases we must have $\gamma(G)=\kappa(G)$, and we obtain a precise formula.

\begin{theorem}\label{thm:equalities}
Let $G$ be a group such that $\SL_n(q)\le G\le \GL_n(q)$.
\begin{enumerate}
\item If $n=p^a$, where $p$ is a prime and $a\in\N$, and if $n>2$, then
\[
\gamma(G) = \kappa(G) = \left(1-\frac{1}{p}\right)\frac{n}{2}+1.
\]
\item If $n=p^aq^b$ where $p$ and $q$ are distinct primes and $a,b\in\N$, then
\[
\gamma(G) = \kappa(G) = \left(1-\frac{1}{p}\right)\left(1-\frac{1}{q}\right)\frac{n}{2}+2.
\]
\item If $n=6p$ where $p$ is a prime, then $\gamma(G) =\kappa(G) = p+2$.
\item If $n=10p$ where $p$ is a prime, then $\gamma(G)=\kappa(G) = 2p+2$.
\end{enumerate}
\end{theorem}

Certain cases of Theorem \ref{thm:equalities} will require independent treatment, as they arise as exceptional cases in the proof of Theorem \ref{thm:lowerbounds}.

\subsection{Linear bounds}
Theorem \ref{thm:upperbounds} (1), Theorem \ref{thm:lowerbounds} (2), and
Theorem \ref{thm:equalities}, taken together, imply that
\begin{equation}\label{eq:linear}
\frac{n}{12} < \kappa(G) \le \gamma(G) \le \frac{n+1}{2},
\end{equation}
for all $n>2$. The upper bound is exact when $n$ is an odd prime. (When $n=2$ it is known that
$\gamma(G)=2$; see \cite{BL}, or the remark after Proposition \ref{prop:pa} below. It is also easy to show that
$\kappa(G)=2$ in this case.) It follows immediately that
\begin{equation}\label{eq:limsup}
\limsup\frac{\gamma(G)}{n}\ =\ \frac{1}{2}.
\end{equation}

The lower bound for $\gamma$ can be improved, as the following theorem indicates.

\begin{theorem}\label{thm:pibound}
Let $G$ be a group such that $\SL_n(q)\le G\le \GL_n(q)$. Then
$n/\pi^2 <\gamma(G)$.
\end{theorem}

From the first part of Theorem \ref{thm:upperbounds} and from Theorem \ref{thm:pibound}, it is easy to show that
\begin{equation}\label{eq:liminf}
\frac{1}{\pi^2}\ \le\ \liminf \frac{\gamma(G)}{n}\ \le\ \frac{1}{6}.
\end{equation}

It follows from the theorems which we have stated, that $\gamma(G)$ and $\kappa(G)$ are bounded above and below by
monotonic functions which grow linearly with $n$.
It appears that the situation for symmetric groups is similar.
It has been announced in \cite[\S1.1]{BLS}, and will be demonstrated in
a forthcoming paper \cite{BPS} now in preparation, that
$\gamma(S_n)$ and $\gamma(A_n)$ are bounded above and below by linear functions of $n$. In fact the numbers
$\gamma(S_n)$ and $\gamma(\GL_n(q))$ seem to be closely related; in all cases where both are known exactly, they
differ by at most $1$. It is not hard to show, and it is worth remarking in this connection, that the upper bounds stated for $\gamma(G)$
in Theorem \ref{thm:upperbounds} are also upper bounds for $\gamma(S_n)$, improving marginally on
those of \cite[Theorem~A]{BP}. It should also be noted that all of our bounds are independent of
the field size $q$.

We establish the upper bounds of Theorem \ref{thm:upperbounds} in Section \ref{sec:coverings}, by exhibiting explicit
normal coverings of the necessary sizes. This builds on work described in \cite{BEGHM}, in which
coverings of $\GL_n(q)$ by proper subgroups are constructed.
The two lower bounds of Theorem \ref{thm:lowerbounds} are proved in Section \ref{sec:lowerbounds}. Both are proved by exhibiting an independent set of classes.
This requires an account of overgroups of certain special elements in $\GL_n(q)$. For such an account we
rely on \cite{GPPS}, which provides a classification of subgroups whose orders are divisible
by primitive prime divisors of $q^d-1$, for all $d>n/2$.
The remaining cases of Theorem \ref{thm:equalities} are
brought together in Section \ref{sec:equalities}. Finally, Theorem \ref{thm:pibound} is established in
Section \ref{sec:pibound}. Its proof relies on work from the
doctoral thesis of Joseph DiMuro \cite{D}, which extends the classification of \cite{GPPS} to cover all $d\ge n/3$.

The classes of subgroups in our normal covering remain distinct, proper and non-trivial in the quotient of $G$ by $Z(G)$.
This is true also of the classes of maximal overgroups which cover the conjugacy classes in our independent sets.
It follows that the bounds which we have stated for $\gamma(G)$ and for $\kappa(G)$ hold equally for $\gamma(G/Z(G))$ and
for $\kappa(G/Z(G))$.

\section{Normal coverings of $G$}\label{sec:coverings}

We shall write $V$ for the space ${\F_q}^n$. Throughout the paper, we assume that $\SL(V)\le G\le \GL(V)$.

We begin by introducing the classes of subgroups which we shall need for our coverings. Proposition \ref{prop:efssss}
below contains standard information about certain subgroups of $\GL_n(q)$, and we shall not prove it here.

\begin{proposition}\label{prop:efssss}
\begin{enumerate}
\item Let $d$ be a divisor of $n$. There exist embeddings of $\GL_{n/d}(q^d)$ into $\GL_n(q)$. All such embeddings are conjugate by elements of
$\SL_n(q)$, and each has index $d$ in its normalizer in $\GL_d(q)$. If $d$ is prime then the normalizer is a maximal subgroup of $\GL_n(q)$.
\item Suppose that $1\le k<n$, and let $U$ be a $k$-dimensional subspace of $V$. Then the set stabilizer $G_U$ of $U$ in $G$ is a
maximal subgroup of $G$. If $W$ is another $k$-dimensional subspace, then $G_U$ and $G_W$ are conjugate in $G$.
\end{enumerate}
\end{proposition}

It will be convenient to have concise notation for these subgroups.

\begin{definition}
\begin{enumerate}
\item We refer to the maximal subgroups of Proposition~\ref{prop:efssss}~(1) as \emph{extension field subgroups} of degree $d$,
and we write $\efs(d)$ for the conjugacy class consisting of the intersections of all such subgroups with the group $G$.
\item We refer to the subgroups of Proposition \ref{prop:efssss} (2) as \emph{subspace stabilizers} of dimension $k$, and we write $\sss(k)$ for
the conjugacy class consisting of all such subgroups.
\end{enumerate}
\end{definition}

The following technical lemma will be useful.

\begin{lemma}\label{lem:subgroups}
\begin{enumerate} \item Suppose that $X\in \GL(V)$, and that $X$ stabilizes a $k$-dimensional
subspace of $V$. Then $X$ stabilizes a subspace whose dimension is $n-k$.
\item Let $X\in \GL(V)$, and let $p$ be a prime dividing $n$. If $X$ lies in no extension
field subgroup of degree $p$, then it stabilizes a subspace of $V$ whose dimension is coprime with $p$.
\end{enumerate}
\end{lemma}
\begin{proof}
\begin{enumerate}
\item Suppose $X$ stabilizes a space $U$ of dimension $k$. Then the
transpose $X^t$ acts on the dual space $V^*$, and stabilizes the annihilator of $U$, which has dimension $n-k$.
\item If $X$ stabilizes no subspace whose dimension is coprime with $p$, then every irreducible divisor of its
characteristic polynomial has degree divisible by $p$, and must therefore split into $p$ factors over $\F_{q^p}$.
Suppose that the elementary divisors of $X$ are $f_1^{a_1},\dots,f_t^{a_t}$. For each $i$, let $g_i$ be an irreducible
factor of $f_i$ over $\F_{q^p}$, and let $Y\in\GL_{n/p}(q^p)$ have elementary divisors $g_1^{a_1},\dots,g_t^{a_t}$.
Then it is not hard to see that any embedding of $\GL_{n/p}(q^p)$ into $\GL_n(q)$ must map $Y$ to a conjugate of $X$.
\end{enumerate}
\end{proof}

We are now in a position to exhibit some normal coverings of $G$.

\begin{lemma}\label{lem:coverings1}
\begin{enumerate}
\item Let $p$ be a prime dividing $n$. Then there is a normal covering $\mathcal{C}_p$ for $G$ given by
\[
\mathcal{C}_p=\{\efs(p)\}\cup\{\sss(k)\mid 1\le k\le n/2,\ p\nmid k\}.
\]
The size of $\mathcal{C}_p$ is
\[
|\mathcal{C}_p| = \left\lfloor\left(1-\frac{1}{p}\right)\frac{n}{2}\right\rfloor+1+\epsilon,
\]
where
\[
\epsilon =\left\{\begin{array}{ll} 1 &  \textrm{if $p=2$ and $n/2$ is odd,} \\
0 & \textrm{otherwise.}\end{array}\right.
\]
This is minimized when $p$ is the smallest prime divisor of $n$.
\item Let $p_1$ and $p_2$ be distinct prime divisors of $n$. Then there is a normal covering $\mathcal{C}_{p_1,p_2}$ for $G$ given by
\[
\mathcal{C}_{p_1,p_2}=\{\efs(p_1),\efs(p_2)\}\cup\{\sss(k)\mid 1\le k< n/2,\ p_1,p_2\nmid k\}.
\]
The size of $\mathcal{C}_{p_1,p_2}$ is
\[
|\mathcal{C}_{p_1,p_2}|=\left(1-\frac{1}{p_1}\right)\left(1-\frac{1}{p_2}\right)\frac{n}{2}+2.
\]
This is minimized when $p_1$ and $p_2$ are the two smallest prime divisors of~$n$.
\end{enumerate}
\end{lemma}
\begin{proof}
The sizes of the sets $\mathcal{C}_p$ and $\mathcal{C}_{p_1,p_2}$ are easily seen to be as stated. That $\mathcal{C}_p$ is a covering follows
immediately from Lemma \ref{lem:subgroups}. So it remains only to prove that $C_{p_1,p_2}$ is a covering.

Let $X\in G$, let $f_X$ be the characteristic polynomial of $X$, and let $g_1,\dots,g_s$ be the irreducible factors
of $f_X$ over $\F_q$, with degrees $d_1,\dots,d_s$ respectively. Then clearly there exist $X$-invariant subspaces
$U_1,\dots,U_s$ such that $\dim U_i = d_i$ for all $i$, and such that $U_i\cap U_j=\{0\}$ whenever $i\neq j$. If any
$d_i$ is divisible by neither $p_1$ nor $p_2$, then $X$ is contained in a subspace stabilizer from one of the classes
in $\mathcal{C}_{p_1,p_2}$. So we assume that each $d_i$ is divisible by at least one of $p_1$ or $p_2$. Suppose that
$d_a$ is
divisible by $p_1$ but not $p_2$, and that $d_b$ is divisible by $p_2$ but not $p_1$. Then $U_a\oplus U_b$ is an
$X$-invariant subspace, and its dimension is coprime with $p_1$ and $p_2$; so again, $X$ is in a subspace stabilizer
from $\mathcal{C}_{p_1,p_2}$. But if no such $d_a$ and $d_b$ can be found, then either all of the $d_i$ are divisible
by $p_1$, or they are all divisible by $p_2$. In this case, $X$ lies in an extension field subgroup either of
degree $p_1$ or of degree $p_2$. \end{proof}

We note that the argument of the last paragraph of this proof does not extend to the case of three primes,
$p_1,p_2,p_3$. It is possible to find matrices whose invariant subspaces all have dimensions divisible
by one of those primes, but which lie in no extension field subgroup. In the case that the primes are $2,3$
and $5$, for instance, there are $30$-dimensional matrices whose irreducible invariant spaces have
dimensions $2$, $3$ and $25$. (Another example is used in the proof of Proposition \ref{prop:10p} below.)
This is the explanation for the appearance of the two smallest prime divisors of $n$
in the first upper bound of Theorem \ref{thm:upperbounds}, which may at first seem a little curious.

The second upper bound of Theorem \ref{thm:upperbounds} is proved in a somewhat similar fashion.

\begin{lemma}\label{lem:coverings2}
Let $p_1, \dots, p_\nu$ be the distinct primes dividing $n$. Then there is a normal covering $\mathcal{D}$ of $G$ given by
\begin{equation*}
\begin{split}
\mathcal{D} &=\{\sss(k)\mid 1\le k\le n/3\} \\
&\quad \cup\ \{\sss(k)\mid n/3< k\le n/2,\ \hcf(k,n)=1\} \\
&\quad \cup\ \{\efs(p_i)\mid 1\le i\le \nu\}.
\end{split}
\end{equation*}
For $n>6$, the size of $\mathcal{D}$ is
\[
\left\lfloor\frac{n}{3}\right\rfloor + \phi(6,2,n) + \nu.
\]
\end{lemma}
\begin{proof}
Let $X\in G$. Suppose that $X$ is reducible, and that its smallest non-trivial invariant subspace has dimension
$k$. If $k> n/3$ then it is not hard to see (for instance, by considering the irreducible factors of the
characteristic polynomial) that $X$ stabilizes at most one other proper non-trivial subspace, of dimension
$n-k$. It follows that if $p$ is a prime dividing both $n$ and $k$, then $X$ is contained in an
element of $\efs(p)$. It is now a straightforward matter to show that $\mathcal{D}$ is a normal covering, and
we omit further details. The size of $\mathcal{D}$ follows immediately from its definition.
\end{proof}

\section{Lower bounds for $\kappa(G)$}\label{sec:lowerbounds}

Recall that $\GL_n(q)$ contains elements of order $q^n-1$, often known as Singer elements.
Such elements stabilize no non-trivial proper subspace of $V$. The determinant of a
Singer element generates the multiplicative group of $\F_q$.

In order to handle all groups $G$ in the range $\SL_n(q)\le G\le \GL_n(q)$ together, we define a parameter
$\alpha\in\N$ as follows.
\[
\alpha = \left\{\begin{array}{ll}
0 & \textrm{if $G=\SL_n(q)$},\\
-|\GL_n(q):G| & \textrm{otherwise}.
\end{array}\right.
\]
Let $\zeta$ be a generator of the multiplicative group of $\F_q$. Then we have
\[
\frac{G}{\SL_d(q)}\cong \langle \zeta^\alpha \rangle.
\]

\begin{definition}
\begin{enumerate}
\item For $d=1,\dots,n$, let $\Gamma_d$ be a Singer element with determinant $\zeta$ in $\GL_d(q)$.
\item For $k< n/2$, define \[\Sigma_k=\diag({\Gamma_k}^{\alpha-1},\Gamma_{n-k}).\]
\item For $j< (n-2)/4$, define \[T_j= \diag({\Gamma_j}^{\alpha-2},\Gamma_{j+1},\Gamma_{n-2j-1}).\]
\end{enumerate}
\end{definition}

The reasons for defining $\alpha$ as above will be clear from the following remark.
\begin{remark}
\begin{enumerate}
\item Since $\det\Sigma_k = \det T_j = \zeta^\alpha$, we have $\Sigma_k, T_j\in G$.
\item It is clear from the definition of $\alpha$ that  $(1-q) < \alpha \le 0$, and hence that
$|\alpha -2|<q+1$. It follows easily that the actions of the matrices ${\Gamma_k}^{\alpha-1}$ and
${\Gamma_j}^{\alpha-2}$ are irreducible for all $k$ and $j$. Therefore the module $\F_q\langle\Sigma_k\rangle$
decomposes into precisely two irreducible summands,
and $\F_q\langle T_j\rangle$ decomposes into precisely three irreducible summands.
\end{enumerate}
\end{remark}

\begin{lemma}\label{lem:overgroups} Suppose that $n>4$. Let $k<n/2$, and if $q=2$ then suppose that $n-k\neq 6$.
Let $j<(n-2)/4$, and if $q=2$ then suppose that $n-2j-1\neq 6$.
\begin{enumerate}
\item If $M$ is a maximal subgroup of $G$ containing $\Gamma_n$ then $M$ is an extension field subgroup of prime degree.
\item If $M$ is a maximal subgroup of $G$ containing $\Sigma_k$ then $M$ is either an extension field subgroup whose degree
is a prime divisor of $\gcd(k,n)$, or else the stabilizer of a subspace of dimension $k$ or $n-k$.
\item Let $n$ have at least $3$ distinct prime divisors. If $M$ is a maximal subgroup of $G$ containing
$T_j$, then $M$ is the stabilizer of a subspace whose dimension is one
of $j$, $j+1$, $2j+1$, $n-2j-1$, $n-j-1$ or $n-j$.
\end{enumerate}
\end{lemma}
\begin{proof} For all but one of the groups $\Gamma_n$ and $\Sigma_k$ under consideration, this is given in
Theorem 4.1 of \cite{BEGHM}; though the result is stated there only for the groups $\GL_n(q)$ and $\SL_n(q)$,
the proof applies equally to intermediate subgroups. The matrices $\Gamma_n$ and $\Sigma_k$ are
there referred to as $\mathrm{GL0}$ and $\mathrm{GLk}$ respectively.

For $\Gamma_n$, the result is essentially that of Kantor \cite{K}. For $\Sigma_k$,
the proof in \cite{BEGHM} relies on the existence of primitive prime divisors of $q^{n-k}-1$, which is given by
Zsigmondy's Theorem \cite{Z} for all pairs $(q,n-k)$ except $(2,4)$ and $(2,6)$; the second of these exceptions
accounts for the excluded case in the statement of the present lemma. The argument uses the classification in
\cite{GPPS}, of subgroups of $\GL_n(q)$ whose order is divisible by a prime divisor of $q^e-1$, where $e>n/2$.

The exceptional case which is not covered in \cite{BEGHM} is the group $\GL_{11}(2)$. In this case we require a
reference directly to the lists of \cite{GPPS}. We find that there are several irreducible subgroups whose order
is divisible by the primitive prime divisor $11$ of $2^{10}-1$; we must show that none of these contains $\Sigma_1$.
All of these subgroups are almost simple, and have a socle which is isomorphic either to one of the Mathieu groups
$\mathrm{M}_{23}$ or $\mathrm{M}_{24}$, or to the unitary group $\PSU_5(2)$, or to a linear group $\SL_2(11)$ or
$\SL_2(23)$. (These subgroups may be found in Table 5 (lines 12 and 14) and Table 8 (lines 2, 7 and 9) of \cite{GPPS}.)
Information about these groups may be found in \cite{ATLAS}. Neither any of the groups themselves, nor any of their
outer automorphism groups, have order divisible by $31$. Therefore an almost simple group of one of these types can
contain no element of order $2^{10}-1=3\cdot 11\cdot 31$, which is the order of the element $\Sigma_1$.

For the groups $T_j$ we refer once again to the classification of \cite{GPPS}, this time for
matrix groups whose order is divisible by a primitive prime divisor of $q^{n-2j-1}-1$. It is not hard to see that
$T_j$ has no overgroups of classical type. The condition that
$n$ has $3$ prime divisors rules out the small dimensional sporadic examples contained in Tables 1--7.
Other examples are ruled out because their order is less than $q^{n-2j-1}-1$, which is the order of the
summand $\Gamma_{n-2j-1}$ of $T_j$.
\end{proof}

We define a set of classes which will help us to establish a first lower bound for $\kappa(G)$.

\begin{definition}
Define a set $\Phi$ of classes of $G$ by
\[
\Phi=\{ \left[\Sigma_p\right]\mid p|n,\ p\textrm{\ prime,\ } p<n/2\}\,\cup\,
\{\left[\Sigma_k\right] \mid k<n/2,\ \hcf(n,k)=1\},
\]
where $[g]$ denotes the conjugacy class of $g$.
\end{definition}

\begin{lemma}\label{lem:setsize1}
Let $n>2$, and let $\nu(n)$ be the number of prime factors of $n$. Then
\[
|\Phi|=\phi(n)/2 + \nu(n) -\epsilon,
\]
where
\[
\epsilon =\left\{\begin{array}{ll} 1 & \textrm{if $n=2p$ for some odd prime $p$,} \\
0 & \textrm{otherwise.} \end{array}\right.
\]
\end{lemma}
\begin{proof} This is immediate from the definition of $\Phi$.\end{proof}

Lemma \ref{lem:setsize1}, together with the following two lemmas, will imply the first part of
Theorem \ref{thm:lowerbounds}.

\begin{lemma}\label{lem:nocommonovergroups1}
$\Phi$ is an independent set of classes.
\end{lemma}
\begin{proof}
Suppose that $q\neq 2$, or that $\left[\Sigma_{n-6}\right]\notin \Phi$. Then Lemma
\ref{lem:overgroups} provides full information about the maximal subgroups
of $G$ which contain elements of $\Phi$, and it is easy to check that the result holds in this case.

Next suppose that $q=2$ and $\left[\Sigma_{n-6}\right]\in \Phi$. (This implies that $n\in \{7,8,9,11\}$.)
Lemma \ref{lem:overgroups} gives full information about the
maximal subgroups of $G$ covering elements of the classes in $\Phi$ other than $\left[\Sigma_{n-6}\right]$.
No class of subgroups contains elements of more than one such class, and it is
easy to check that none covers the element $\Sigma_{n-6}$ itself.\end{proof}

\begin{lemma}\label{lem:lowerboundexception}
Let $n=2p$ where $p>2$ is a prime. Then $\kappa(G)\ge|\Phi|+1$.
\end{lemma}
\begin{proof}
The proof of Lemma \ref{lem:nocommonovergroups1} shows that in any normal covering of $G$, the
distinct classes in $\Phi$ are covered by distinct classes of subgroups.
We add an extra conjugacy class to $\Phi$, namely the class represented by
$\Sigma_p=\diag({\Gamma_p}^{\alpha-1},\Gamma_p)$, where
$\Gamma_p$ is a Singer element in $\GL_p(q)$. This
element stabilizes no subspace of dimension $k$ for any $k$ coprime with $n$; nor does it stabilize a
subspace of dimension $2$ or $n-2$. Therefore, by the second part of Lemma \ref{lem:overgroups},
if there is a covering of $G$ of size $|\Phi|$, then $\Sigma_p$ must lie in a subgroup in $\efs(2)$.

Note that since $2$ and $p$ are coprime, ${\Sigma_p}^2$ has two irreducible summands of dimension
$p$. It is not hard to show that these submatrices are not conjugate, and neither of them is reducible over $\F_{q^2}$;
it follows that ${\Sigma_p}^2$ is not contained in any embedding of $\GL_p(q^2)$ into $G$. Hence $\Sigma_p$
itself is not contained in an embedding of \mbox{$\GL_p(q^2)\cdot 2$}. \end{proof}

This completes the proof of the first part of Theorem \ref{thm:lowerbounds}.
\medskip

We define a second independent set of classes which yields the second lower bound of Theorem \ref{thm:lowerbounds}.
We shall require the following lemma.

\begin{lemma}\label{lem:Bertrand}
Let $p$ be a prime divisor of $n$. Suppose that $n$ has at least $3$ distinct prime divisors, and that $n$ is not
equal to $6q$ or $10q$ for any prime $q$. Then there exists an integer $w_p$ such that $(n-2)/4 \le w_p < n/2$, and
such that $w_p$ is divisible by $p$, and by no other prime divisor of $n$. If $p\neq 3$
then $w_p$ may be chosen so that it is not divisible by $3$.
\end{lemma}
\begin{proof}
Bertrand's Postulate states that for every $k>3$ there is a prime $r$ such that $k<r<2k-2$.
The conditions on $n$ imply that $n\ge 12p$. So there is a prime $r>3$ such that
\[
\frac{n}{4p} < r < \frac{n}{2p}.
\]
If $r$ is not itself a prime divisor of $n$, or if it is equal to $p$, then we may take $w_p = pr$.
On the other hand, if $r$ is a prime divisor of $n$ other than $p$ then clearly $n=3pr$, and since we have assumed
that $n\ge 12p$, we have $r\ge 5$. Now we see that there exists $m$ equal either to $r+1$ or to $r+2$, such that
$m$ is not divisible by $3$, and we may take $w_p=pm$.
\end{proof}

\begin{definition}
Let $n$ be a number with at least $3$ prime divisors, and not equal to $6p$ or $10p$
for any prime $p$. We define a set $\Psi$ of classes of $G$ by
\begin{equation*}
\begin{split}
\Psi &=\{\left[T_j\right]\mid j<(n-2)/4, j\equiv 1 \bmod 3 \} \\
&\quad  \cup\ \{\left[\Sigma_k\right] \mid n/4 < k < n/2,\ \hcf(3n,k)=1\}\\
&\quad  \cup  \{\left[\Sigma_{6b}\right] \mid b < n/12, \hcf(n,6b)=1\}\\
&\quad  \cup\ \{\left[\Sigma_{w_p}\right]\mid p|n, \textrm{\ $p$ prime}\},
\end{split}
\end{equation*}
where $w_p$ is as constructed in Lemma \ref{lem:Bertrand}, and where $[g]$ denotes the conjugacy class of $g$.
\end{definition}

To describe the size of the set $\Psi$ we use Lehmer's partial totient function $\phi(k,t,n)$, which was defined
before the statement of Theorem \ref{thm:upperbounds} above.

\begin{lemma}
Let $n$ have $\nu$ distinct prime divisors, where $\nu\ge 3$, and suppose that $n$ is not equal to $6p$ or $10p$
for any prime $p$.
\begin{enumerate}
\item If $2$ or $3$ divides $n$, then
\[
|\Psi| = \left\lfloor\frac{n+6}{12}\right\rfloor + \phi(12,1,3n) + \nu.
\]
\item If $\hcf(n,6)=1$, then
\[
|\Psi| = \left\lfloor\frac{n+6}{12}\right\rfloor + \phi(12,1,3n) + \phi(12,0,n) + \nu.
\]
\end{enumerate}
\end{lemma}
\begin{proof}

We write $\lceil x \rceil$ for the least integer not less than $x$.
The size $X$ of the set $\{\left[T_j\right]\mid j<(n-2)/4, j\equiv 1 \bmod 3 \}$ is
$\lceil N/3\rceil$, where $N=\lfloor (n-2)/4\rfloor$. By examining residues modulo $12$, it is not hard to show
that $X=\lfloor (n+6)/12\rfloor$, the first term in our sum.

The set $\{\left[\Sigma_k\right] \mid n/4 < k < n/2,\ \hcf(3n,k)=1\}$ clearly has size $\phi(12,1,3n)$,
immediately from the definition of the function $\phi(k,t,n)$. The set
$\{\left[\Sigma_{6b}\right] \mid b < n/12, \hcf(n,6b)=1\}$ is empty if $\hcf(n,6)\neq 1$; otherwise
it has size $\phi(12,0,n)$. And clearly the set
$\{\left[\Sigma_{w_p}\right]\mid p|n, \textrm{\ $p$ prime}\}$ has size $\nu$ as required.\end{proof}

To establish the second lower bound in Theorem \ref{thm:lowerbounds}, it will suffice to show that any normal
covering for $G$ has size at least $|\Psi|$. This is done in the following lemma.

\begin{lemma}\label{lem:nocommonovergroups2}
Let $n$ have at least $3$ distinct prime divisors, and not equal to $6p$ or $10p$ for any prime $p$.
Then $\Psi$ is an independent set of classes.
\end{lemma}
\begin{proof}
Lemma \ref{lem:overgroups} describes the maximal subgroups
of $G$ which contain elements of the classes in $\Psi$. The elements $T_j$ lie only in members of
$\sss(\ell)$ or $\sss(n-\ell)$, where $\ell\in\{j,j+1,2j+1\}$. Notice that if $\ell>n/4$ then $\ell=2j+1$, and
hence $\ell\equiv 3\bmod 6$. The elements
$\Sigma_k$, where $k$ is coprime with $n$, lie only in members of $\sss(k)$ or
$\sss(n-k)$. And the elements $\Sigma_{w_p}$ lie in subspace stabilizers and also in elements of $\efs(p)$. It is easy to check that the
values permitted for $j$, $k$, $b$ and $w_p$ ensure that no two elements of distinct classes in $\Psi$ stabilize subspaces of
the same dimension. Therefore no two classes in $\Psi$ can be covered by a single class of subgroups.\end{proof}

\section{Several equalities}\label{sec:equalities}

In this section we establish the various claims of Theorem \ref{thm:equalities}.
We do this simply by comparing upper and lower bounds from earlier parts of the paper.

\begin{proposition}\label{prop:pa}
If $n=p^a$, where $p$ is a prime and $a\in\N$, and if $n>2$, then
\[
\gamma(G) = \kappa(G) = \left(1-\frac{1}{p}\right)\frac{n}{2}+1.
\]
\end{proposition}
\begin{proof}
Lemma \ref{lem:coverings1} and Lemma \ref{lem:nocommonovergroups1} together tell us that
\[
|\Phi|\le\kappa(G)\le\gamma(G)\le|\mathcal{C}_p|.
\]
But it is easy to check, using Lemma \ref{lem:setsize1},
that $|\Phi|=|\mathcal{C}_p|$, and that this number is as claimed in the proposition.
\end{proof}

\begin{remark}
If $n=2$, then the covering $\mathcal{C}_2$ has size $2$. Since no finite group is covered by a single class of
proper subgroups, it follow that $\gamma(G)=2$ in this case.
\end{remark}

\begin{proposition}\label{prop:paqb}
If $n=p^aq^b$ where $p$ and $q$ are distinct primes and $a,b\in\N$, then
\[
\gamma(G) = \kappa(G)= \left(1-\frac{1}{p}\right)\left(1-\frac{1}{q}\right)\frac{n}{2}+2.
\]
\end{proposition}
\begin{proof}
As in the proof above, Lemma \ref{lem:coverings1} with Lemmas \ref{lem:nocommonovergroups1} and
\ref{lem:lowerboundexception} yield that
\[
|\Phi|+\epsilon\le\kappa(G)\le\gamma(G)\le|\mathcal{C}_{p,q}|,
\]
where $\epsilon=1$ if $n=2p$ (or $n=2q$), and $\epsilon=0$ otherwise.
But we see that $|\Phi|+\epsilon=|\mathcal{C}_{p,q}|$,
with this number being as claimed in the proposition.
\end{proof}

\begin{proposition}\label{prop:6p}
If $n=6p$ where $p$ is a prime, then
\[
\gamma(G) = \kappa(G) = p+2.
\]
\end{proposition}
\begin{proof}
In this case we have
\[
|\Phi|\le\kappa(G)\le\gamma(G)\le|\mathcal{C}_{2,3}|,
\]
and it is easy to calculate that $|\Phi|=|\mathcal{C}_{2,3}|=p+2$.
\end{proof}

\begin{proposition}\label{prop:10p}
If $n=10p$ where $p$ is a prime, then
\[
\gamma(G) = \kappa(G) = 2p+2.
\]
\end{proposition}
\begin{proof}
If $p$ is $2$ or $5$ then the result follows from Proposition \ref{prop:paqb}; if $p=3$ then it follows
from Proposition \ref{prop:6p}. So we may assume that $p>5$.
Then we have
\[
|\Phi|\le\kappa(G) \le \gamma(G)\le|\mathcal{C}_{2,5}|,
\]
but in this case we see that $|\Phi|=2p+1$ whereas $|\mathcal{C}_{2,5}|=2p+2$. To prove that the upper bound is sharp
for $\kappa(G)$, it will be sufficient to exhibit an element $Y$ of $G$ which cannot be covered by any class of subgroups containing
an element of any conjugacy class in $\Phi$. We define
\[
Y= \diag({\Gamma_p}^{\alpha-2},\Gamma_5,\Gamma_{n-p-5}).
\]
Notice that $n-p-5$ is even, and coprime with $5$ and with $p$. It follows that $Y$
does not stabilize a subspace of dimension coprime with $n$. But certainly $Y$ lies in no extension field subgroup,
and so it satisfies the required condition.
\end{proof}

\section{Proof of Theorem 1.4}\label{sec:pibound}

For a positive integer $n$, let $f(n)$ be the number of partitions
of $n$ with exactly three parts. By an elementary counting
argument the following formula can be found for $f(n)$.

\begin{lemma}\label{lemma:f(n)}
\[
\displaystyle{f(n) =\left\{
\begin{array}{ll}
\frac{1}{12}(n-1)(n-2) + \frac{1}{2}\lfloor (n-1)/2 \rfloor & \textrm{if\ } 3 \nmid n, \\ \\
\frac{1}{12}(n-1)(n-2) + \frac{1}{2}\lfloor (n-1)/2 \rfloor +
\frac{1}{3} & \textrm{if\ } 3 \mid n.
\end{array}\right.}
\]
\end{lemma}

It follows from Lemma \ref{lemma:f(n)} that
\[
\left|f(n) - \frac{n^{2}}{12}\right| \le \frac{1}{3}.
\]
We define $\epsilon_n = f(n)-n^2/12$.

Let $P(n)$ be the set of partitions of $n$ into three parts having no common divisor greater than $1$.
Let $g(n) = |P(n)|$. Then we have $f(n) = \sum_{d \mid n}
g(d)$. By the M\"obius Inversion Formula, we obtain
\begin{align*}
g(n) &= \sum_{d \mid n} \mu(d) f(n/d) = \left( \sum_{d \mid n}\mu(d) \frac{1}{12}{(n/d)}^{2} \right) + \left( \sum_{d \mid n} \mu(d)\epsilon_{n/d} \right)\\
&> \frac{n^{2}}{12} \left(\sum_{d \mid n} \frac{\mu(d)}{d^2} \right) + \left( \sum_{d \mid n} \mu(d) \epsilon_{n/d} \right)\\
&> \frac{n^2}{12} \left(\prod_{ p \textrm{\ prime}} \left(1 - \frac{1}{p^2}\right)\right) + \left( \sum_{d \mid n} \mu(d) \epsilon_{n/d} \right).
\end{align*}
Since
\begin{align*}
\prod_{ p \textrm{\ prime}} \left(1 - \frac{1}{p^2}\right) &= \left( \prod_{p \textrm{\ prime}} (1 + p^{2} + p^{4} + \ldots) \right)^{-1} \\
&= \left( \sum_{n=1}^{\infty} \frac{1}{n^2} \right)^{-1}\\
&= \frac{6}{\pi^2},
\end{align*}
we have
\[
g(n) >  \left(\frac{n^2}{2 \pi^2}\right) + \left( \sum_{d \mid n} \mu(d) \epsilon_{n/d}\right).
\]
Now since the number of divisors of $n$ is less than $2\sqrt{n}$, we obtain the following lemma.

\begin{lemma}\label{lemma:g(n)}
\[
\frac{n^2}{2 \pi^2} - \frac{2}{3}\sqrt{n} < g(n).
\]
\end{lemma}

The next lemma is the principal step in our proof. It gives information about the maximal overgroups in
$G$, of an element of the form $\diag({\Gamma_a}^{\alpha-2},\Gamma_{b},\Gamma_{c})$, where $a$, $b$ and $c$ are coprime.
The proof relies on knowledge of the subgroups of $\GL_n(q)$ whose order is divisible by a primitive prime divisor of
$q^d-1$, where $d> n/3$. An account of such subgroups is given in the doctoral dissertation of Joseph DiMuro \cite{D};
this work extends the classification of \cite{GPPS},
which deals with the case $d>n/2$.

\begin{lemma}\label{lem:overgroups2}
Let $\nu(n) \ge 3$ and let $n \ge 98$. For $\lambda = (a,b,c) \in P(n)$, with $a \le b \le c$, and with
$a,b,c$ coprime,
let $g=g_{\lambda} = \diag({\Gamma_a}^{\alpha-2},\Gamma_{b},\Gamma_{c})$.
Then every maximal overgroup $M$ of $g$ in $G$ is a subspace stabilizer,
except possibly in the following cases.
\begin{enumerate}
\item[(i)] $2|n$, $c = n/2$, and $M \cong G \cap (\GL_{n/2}(q) \wr
C_{2})$;
\item[(ii)] $4|n$, $(a,b,c) =(2,(n-2)/2,(n-2)/2)$, and either $M \cong G \cap
(\GL_{n/2}(q) \wr C_{2})$, or $M \cong G \cap (\GL_{n/2}(q) \circ
\GL_{2}(q))$. (Here $\circ$ is used to denote a central product.)
\end{enumerate}
\end{lemma}

\begin{proof}
We observe that $V$ may be decomposed as $V_a\oplus V_b\oplus V_c$, where $V_a$, $V_b$ and $V_c$ are
$g$-invariant subspaces of dimensions $a$, $b$ and $c$ respectively. The action
of $g$ on each of these summands is irreducible.
It follows that $g$ lies in the stabilizers of proper subspaces of at least $4$ different dimensions; and so $g$
is covered by the class $\sss(k)$ for at least $4$ values of $k$.

Note that $c > n/3$, and that $q^{c}-1$ divides the order of $g$.
Hence a maximal overgroup $M$ of $g$ must belong to one of the classes of groups
mentioned in Section 1.2 of \cite{D}. We observe firstly that
owing to our assumption that $\nu \ge 3$ and $n
\ge 98$, the subgroup $M$ cannot be any of those in Tables 1.1--1.9 of \cite{D};
this immediately rules out several of the Examples listed there.
We shall go through the remaining Examples.
\medskip

\noindent Example 1. \emph{Classical examples}. The determinant of $g$ is a generator of the quotient $G/\SL_n(q)$,
and so $M$ cannot contain $\SL_n(q)$.

Any element of a symplectic or orthogonal group is similar to its own inverse; an element $g$ of a unitary
group is similar to its conjugate-inverse $g^{-\tau}$, where $\tau$ is induced by an involutory field automorphism.
(See \cite{W}, Section 2.6, or (3.7.2) for groups in characteristic $2$.)

If $M$ normalizes a symplectic or orthogonal group $H$, then $g^{q-1}$ lies in $H$ itself, and so
$g^{q-1}$ is similar to its own inverse. Then it is clear that
${\Gamma_c}^{q-1}$ is similar to its own inverse (it does not matter here whether or not $b=c$).
But this cannot be the case since $c>2$.

Similarly, if $M$ normalizes a unitary group $U$ then
$g^{q+1}$ lies in $U$, and it follows that $g^{q+1}$ is similar to its conjugate-inverse. But then it
follows that ${\Gamma_c}^{q+1}$ is similar to its conjugate-inverse, and it is easy to show that this
is not the case.
\medskip

\noindent Example 3. \emph{Imprimitive examples.} Here $M$
preserves a decomposition $V = U_{1} \oplus \cdots \oplus U_{t}$ for $t \ge 2$.
Let $\dim U_{i} = m$, so that $n = mt$. Recall that the $\langle g \rangle$-module $V$ is the
direct sum of $3$ irreducible submodules $V_{a}$, $V_{b}$, $V_{c}$
of dimensions $a$, $b$, $c$ respectively. So $\langle g\rangle$ has at most $3$ orbits on the set of spaces $U_{i}$.


Let $r$ be the smallest integer such that $V_c$ is contained in the direct sum
of $r$ of the spaces $U_{i}$. We observe that $n/3 < c\le rm$, and so $m>n/3r$. Without loss of generality,
we may assume
that $V_{c} \leq W = U_{1} \oplus \cdots \oplus U_{r}$. It is clear that $W$ is $g$-invariant.
Let $\overline{g}$ be the restriction of $g$ to $W$.
Then $\langle \overline{g} \rangle$ acts transitively on $\{ U_{1}, \ldots , U_{r}
\}$. Since $\overline{g}^{\,r}$ acts in the same way on each $U_i$ for $i\le r$,
an upper bound for the order of $\overline{g}$ is $(q^{m}-1)r$.
But since $m\le n/r$, and since $n\ge 98$ by assumption, we see that $(q^m-1)r<q^{n/3}-1$ if $r \ge 4$.
Therefore we must have $r\le 3$.

It follows that $V_c$ is a simple $\F_q\langle \overline{g}^{\,r}\rangle$-module.
Now since $\overline{g}^{\,r}$ commutes with the projections of $W$ onto its summands $U_i$,
we see that at least one of the spaces $U_i$ contains an $\overline{g}^{\,r}$-invariant
subspace of dimension $c$. So $m>n/3$, and hence $r\neq 3$.

Suppose that $r=2$. Since $\overline{g}^{\,r}$ has two fixed spaces of dimension $m$, we see that $b = c = m$, and that
$V_{b} \oplus V_{c} \leq W$. If $W < V$, then $W = V_{b} \oplus V_{c}$. Now we see that $m$ divides each of $a$ and
$b+c=2c$. Since $a$, $b$, $c$ are coprime, it follows that $m=2$. But this implies that $n<6$, which contradicts
the assumption that $n\ge 98$. So we may suppose that $W = V$. Then it is not hard to show that
$V_a$ has two irreducible summands as an $\langle \overline{g}^{\,2} \rangle$-module.
But this can occur only when $a=2$, and this accounts for the first of the exceptional cases of the lemma.

Finally, if $r=1$ then $m\ge c > n/3$, and so $t=2$. It is easy to see, in this case, that $c=m=n/2$, and this accounts
for the second exceptional case of the lemma.

\medskip

\noindent Example 4. \emph{Extension field examples.} If $g$ stabilizes an $\F_{q^r}$-structure
on $V$, then $g^r$ lies in the image of an embedding of $\GL_{n/r}(q^r)$ into $\GL_n(q)$. Now if this is the
case then it is not hard, by considering the degrees of the eigenvalues
of $g$ over the fields $\F_q$ and $\F_{q^r}$, to show that $r$ must divide each of $a,b,c$. But this implies
that $r=1$, since $a,b,c$ are coprime.
\medskip

\noindent Example 5. \emph{Tensor product decomposition examples.} Here $M$ stabilizes a non-trivial tensor
product decomposition $V=V_1\otimes V_2$. The central product $\GL(V_1)\circ \GL(V_2)$ embeds into $\GL_n(q)$,
and $M$ is the intersection of this group with $G$. For $x_1\in \GL(V_1)$ and $x_2\in\GL(V_2)$,
we write $(x_1, x_2)$ for the corresponding element of
$\GL(V_1)\circ\GL(V_2)$.

We shall suppose that $V_1$ and $V_2$ have dimensions~$n_1$
and~$n_2$ respectively, with $n_1\le n_2$. Then since $c>n/3$, it is not hard to see that we have
$n_1 = 2$.


Suppose that $g\in M$, and let $g_1\in \GL(V_1)$ and $g_2\in\GL(V_2)$ be such that $g=(g_1,g_2)$.
Let $h=g^{q^2-1}$.
Since the order of $g$ is coprime with $q$, we see that $g_1^{q^{2}-1}$ is the identity on $V_1$, and so $h=(1,h_2)$ for
some $h_2\in\GL(V)$.

The largest dimension of an irreducible
$\langle h \rangle$-subspace of $V$ is $c$, and there are at most $2$
such subspaces. We obtain the $\langle h \rangle$-subspace
decomposition of $V$ up to isomorphism by taking two copies of each summand
of the $\langle h_2 \rangle$-subspace decomposition of $V_{2}$. It follows that there must be at least two
summands of dimension $c$, and hence that
$b=c$ and that $a < b$. It follows also that the $a$ dimensional summand of $g$ splits into two summands
as an $\F_q\langle h\rangle$-module. But it is not hard to see that this can occur only if
$a = 2$, and so we have $a = 2$ and $b = c = (n-2)/2$. This is the second exceptional case of the lemma.
\medskip

\noindent Example 6. \emph{Subfield examples.} These cannot occur, since $g$ is built
up using Singer cycles, which do not preserve any proper subfield structure.
\medskip

\noindent Example 7. \emph{Symplectic type examples.} This class of groups exists only in prime-
power dimension, and cannot occur in the cases we are considering since
we have assumed that $\nu\ge 3$.
\medskip

\noindent Example 8.(a) \emph{Permutation module examples.}
In this case $S$ is an alternating group $A_m$ for some $m
\ge 5$. Then it is known that the order of an element
in $M$ is at most $(q-1)\cdot e^{\theta \sqrt{m \log m}}$ where $\theta =
1.05314$, by a result of Massias \cite{M}. Here $n =
m-1$ or $m-2$. But a routine calculation shows that the inequality  $e^{\theta \sqrt{(n+2) \log (n+2)}} <
(q^{n/3}-1)/(q-1)$ holds for all $q\ge 2$, and for all $n \ge 98$. (This inequality fails when $q=2$ and $n=97$.)
\medskip

\noindent Example 11. \emph{Cross-characteristic groups of Lie type.}
The examples not yet ruled out are contained in Table 1.10 of \cite{D}.
But every element of $M$ has order less than $n^{3}$, which is less than $q^{n/3}-1$ for $n \ge 98$.
\end{proof}
\bigskip

We are now in a position to complete the proof of Theorem \ref{thm:pibound}.

\begin{proof}
Define a set $\Omega$ of classes of $G$ by
\[
\Omega = \{[\Gamma_{n}^{\alpha+q-1}]\} \cup \{[g_{\lambda}] : \lambda \in P(n) \}.
\]
Let $\mathcal{C}$ be a set of conjugacy classes of subgroups of $G$ which covers $\Omega$,
of the smallest size such that this is possible.
Then clearly $|\mathcal{C}|\le \gamma(G)$.
By the theorem of Kantor \cite{K} mentioned in the proof of Lemma \ref{lem:overgroups} above, and
by Lemma \ref{lem:overgroups2}, we see that $\mathcal{C}$ must contain a single class of extension field subgroups.
If $n\ge 98 $ and $\nu\ge 3$ then each remaining elements of $\mathcal{C}$ is either a
class of subspace stabilizers, or else one of the classes of subgroups mentioned in
the exceptional cases of Lemma \ref{lem:overgroups}.
Each subspace stabilizer contains at most $n/2$ of the elements $g_{\lambda}$, and each of the exceptional classes contains
at most $n/4$. Now, using Lemma \ref{lemma:g(n)}, we see that
\[
\gamma(G)\ \ge\ |\mathcal{C}|\ \ge\ 1+\frac{2g(n)}{n}\ >\ \frac{n}{\pi^2},
\]
as required for the theorem.

To remove the conditions that $n\ge 98 $ and that $\nu\ge 3$, it is enough to observe that the lower
bound for $\kappa(G)$ given by Theorem \ref{thm:lowerbounds} is larger than $n/\pi^{2}$ in any case where
either of these conditions fails.
\end{proof}

\end{document}